\documentclass[11pt]{article}
\usepackage{amsmath}
\usepackage{amssymb}
\usepackage{graphicx}
\usepackage{float}
\usepackage{ulem}
\usepackage{color}

\newtheorem{theorem}{Theorem}

\newcommand{\Z}{{\mathbb Z}}

\newtheorem{cor}[theorem]{Corollary}

\newcommand{\textdef}{\textit}

\newenvironment{proof}{\noindent{\bf Proof.\,}}{\hfill$\Box$}

\begin{document}

\title{Vertex identifying codes for the $n$-dimensional lattice}
\author{Brendon Stanton\thanks{The author's research partially supported by NSF grant DMS-0901008.}\\Department of Mathematics and Computer Science\\The Citadel\\Charleston, SC 29409}
\maketitle
\begin{abstract}An $r$-identifying code on a graph $G$ is a set $C\subset V(G)$ such that for every vertex in $V(G)$, the intersection of the radius-$r$ closed neighborhood with $C$ is nonempty and different.   Here, we provide an overview on codes for the $n$-dimensional lattice, discussing the case of 1-identifying codes, constructing a sparse code for the 4-dimensional lattice as well as showing that for fixed $n$, the minimum density of an $r$-identifying code is $\Theta(1/r^{n-1})$.
\end{abstract}
\maketitle



\section{Introduction}

Vertex identifying codes were introduced by Karpovsky, Chakrabarty, and Levitin in~\cite{Karpovsky1998} as a way to help with fault diagnosis in multiprocessor computer systems.  Amongst the many results in that paper, an interesting result is that if $n=2^k-1$ for some integer $k$, we can find a code of optimal density for the $n$-dimensional lattice by using a Hamming code.  Denote by $\mathcal{D}(G,r)$ the minimum possible density of an $r$-identifying code for a graph $G$.  Let $L_n$ denote the $n$-dimensional lattice.  In ~\cite{Stanton2011a}, we present a slight generalization of this proof.

\begin{theorem}[\cite{Stanton2011a}]\label{thm:domsettheorem}
  Let $D$ be a dominating set for the $n$-dimensional hypercube, then $\mathcal{D}(L_n,1)\le |D|/2^n$.
\end{theorem}

The proof of this comes from replacing Hamming Codes (which are already dominating sets) with the more general dominating sets to get bounds in the case that $n\neq 2^k-1$.  For small values of $n$, we use Table 6.1 of \cite{Cohen1997} to get good bounds in Figure~\ref{fig:boundtable}.

\begin{figure}[ht]
$$\begin{array}{|ccccc|}
\hline
    && \mathcal{D}(L_1,1)&=&1/2  \\
    && \mathcal{D}(L_2,1)&=&7/20^\text{~\cite{Ben-Haim2005}}  \\
    \hline
    && \mathcal{D}(L_3,1)&=&1/4    \\
  1/5&\le& \mathcal{D}(L_4,1) &\le& 2/9^{\text{[Theorem~\ref{thm:4Dcode}]}} \\
  \hline
  1/6&\le& \mathcal{D}(L_5,1) &\le& 7/32^{\text{[Theorem~\ref{thm:domsettheorem}}]}  \\
    1/7&\le& \mathcal{D}(L_6,1) &\le& 3/16^{\text{[Theorem~\ref{thm:domsettheorem}}]}   \\
    \hline
    && \mathcal{D}(L_7,1)&=&1/8    \\
  1/9&\le& \mathcal{D}(L_8,1) &\le& 1/8  \\
  \hline
  1/10&\le& \mathcal{D}(L_9,1) &\le& 31/256^{\text{[Theorem~\ref{thm:domsettheorem}}]}  \\
  1/11&\le& \mathcal{D}(L_{10},1) &\le& 15/128^{\text{[Theorem~\ref{thm:domsettheorem}}]}  \\
  \hline
\end{array}$$
\caption{A table of bounds of densities of codes for small values of $n$.  All bounds not cited are due to \cite{Karpovsky1998}.}\label{fig:boundtable}
\end{figure}

The result for $L_4$ is proven in Section~\ref{4Dsection}.  For larger values of $n$ we use this theorem in conjunction with a result of Kabatyanski{\u\i} and Panchenko\cite{Kabatyanskiui1988} to get a good asymptotic bound.

\begin{cor}\label{cor:asymptotic}  There is a constant $b$ such that for sufficiently large $n$:
$$\frac{1}{n+1}\le \mathcal{D}(L_{n},1) \leq \left( 1 + \frac{b \ln \ln n}{\ln n} \right) \frac{1}{n+1}. $$
\end{cor}

Finally, in  Section~\ref{section:generalbounds}, we prove both an upper and lower bound for $\mathcal{D}(L_n,r)$ and show:

\begin{theorem}\label{thm:thetabound}
  For fixed $n$, $\mathcal{D}(L_n,r)=\Theta(1/r^{n-1})$ as $r\rightarrow\infty$.
\end{theorem}

Given a connected, undirected graph $G=(V,E)$, we define
$B_r(v)$, called the ball of radius $r$ centered at $v$ to be
$$B_r(v)=\{u\in V(G): d(u,v)\le r\}. $$

We call any nonempty subset $C$ of $V(G)$ a \textdef{code} and its elements \textdef{codewords}.  A code $C$ is called \textdef{$r$-identifying} if it has the properties:
\begin{enumerate}
\item $B_r(v) \cap C \neq \emptyset$ for all $v$
\item $B_r(u) \cap C \neq B_r(v)\cap C$, for all $u\neq v$
\end{enumerate}
When $r=1$ we simply call $C$ an identifying code.  When $C$ is understood, we define $I_r(v)=I_r(v,C)=B_r(v)\cap C$.  We call
$I_r(v)$ the identifying set of $v$.  If $I_r(u)\neq I_r(v)$ for some $u\neq v$, the we say $u$ and $v$ are \textdef{distinguishable}.  Otherwise, we say they are \textdef{indistinguishable}.

We formally define the $n$-dimensional lattice $L_n=(V,E)$ where $$V=\Z^n,\qquad
E = \left\{\{(x_1,\ldots,x_n),(y_1,\ldots,y_n)\}: \sum_{i=1}^n |x_i-y_i| = 1\right\}.$$

The density of a code $C$ for a finite graph $G$ is defined as $|C|/|V(G)|$.  Let $Q_m$ denote the set of vertices $(x_1, \ldots, x_n)\in
\Z^n$ with $|x_i|\le m$ for all $1\le i \le n$.  We define the
density $D$ of a code $C$ in $L_n$ similarly to how it is defined in~\cite{Charon2002} by
$$D=\limsup_{m\rightarrow\infty}\frac{|C\cap Q_m|}{|Q_m|}.$$

\section{The 4-dimensional case}\label{4Dsection}

The king grid, $G_K$, is defined to be the graph on vertex set $\Z\times\Z$ with edge set $E_K=\{\{u,v\}: u-v\in\{(0,\pm1),(\pm1,0),(1,\pm1),(-1,\pm1)\}\}$.

\begin{theorem}\label{thm:4Dcode}
 $\mathcal{D}(L_4,1)\le 2/9$.
\end{theorem}
\begin{proof}
  The idea of our proof is to take our original copy of a code for the king grid and copy it to two-dimensional cross-sections of $L_4$--shifting it when ``up and to the right'' when moving in the $x_3$ direction and ``up and to the left'' when moving in the  $x_4$ direction.

  Let $C$ the identifying code of density 2/9 for the king grid given by Cohen, Honkala, Lobstein and Z\'{e}mor in~\cite{Cohen2001}.  For the remainder of this proof, let $B_1^{G}(v)$ denote the ball of radius 1 in the graph $G$ and likewise, let $I_1^{G}(v)$ denote the identifying set of $v$ in $G$.

  Let $v\in V(L_4)$.  Since $(1,0,0,0),(0,1,0,0),(1,1,1,0),$ and $(1,-1,0,1)$ are linearly independent, we may write uniquely write $v=(x,y,0,0)+i(1,1,1,0)+j(1,-1,0,1)$.  Next, we define $\varphi: V(L_4)\rightarrow V(G_K)$ by $\varphi(v)= (x,y)$ and then define $$C'=\{v\in V(L_4):\varphi(v)\in C\}.$$  Fixing, $i$ and $j$, we see that $C'$ consists of isomorphic copies of $C$ and so $C'$ has the same density as $C$.

  It is easy to check that $\varphi(B_1^{L_4}(v)) = B_1^{G_K}((x,y))$.  For instance, $\varphi(v+ (1,0,0,0))= (x+1,y)\in B_1^{G_K}(v)$ and $\varphi (v+ (0,0,1,0)) = \varphi((x-1,y-1,0,0)+(i+1)(1,1,1,0)+j(1,-1,0,1)) = (x-1,y-1)\in B_1^{G_K}(x,y)$.  This shows two things.  First, each vertex has a non-empty set, since $|I_1^{L_4}(v)| = |I_1^{G_K}(\varphi(v))|\ge 1$.  Secondly, it shows that if $\varphi(u)\neq \varphi(v)$, then $u$ and $v$ are distinguishable.  Hence, we only need to distinguish between vertices where $\varphi(u)=\varphi(v)$.

  Without loss of generality, let $u=(x,y,0,0)$ and $u=(x,y,0,0) + i(1,1,1,0) + j(1,-1,0,1)$ and so $$d(u,v) = |i + j|+|i-j|+|i|+|j|.$$  If $i$ and $j$ are both non-zero, then either $|i+j|$ or $|i-j|$ is at least 1 and so $d(u,v)\ge 3$.  If $j=0$, then $d(u,v)=3|i|\ge 3$ and likewise if $i=0$ then $d(u,v)\ge 3|j|\ge 3$.  Since $d(u,v)\ge 3$ in all cases, we only need to consider $I_1^{L_n}(v)$.  It is nonempty and doesn't intersect with $B_1^{L_n}(u)$.  Thus, $u$ and $v$ are distinguishable, completing the proof.
\end{proof}

\section{General Bounds and Construction}\label{section:generalbounds}

We finally wish to produce some general bounds for $r$-identifying codes on the $L_n$.  We start with a lower bound proof, in the style of Charon, Honkala, Hudry and Lobstein\cite{Charon2001}.  First, we define $b_k^{(n)} = |B_k(v)|$ for $v\in V(L_n)$.

\begin{theorem}\label{thm:rlowerbound}
  The minimum density of an $r$-identifying code for $L_n$ is at least $$\mathcal{D}(L_n,r)\ge\frac{(n-1)!\lceil \log_2(2n+1)\rceil}{2^{n+1}r^{n-1} + p_{n-2}(r)}$$ where $p_{n-2}(r)$ is a polynomial in $r$ of degree no more than $n-2$.
\end{theorem}
\begin{proof}
  Let $v\in V(L_n)$ and $u_1,u_2,\ldots, u_{2n}$ be its neighbors.  If $d(v,x)>r+1$, then it is easy to see that $d(u_i,x)\ge r+1$ for all $i$.  Likewise, it is easy to check that if $d(v,x)\le r-1$, then $d(u_i,x)\le r$ for all $i$.  In other words, all vertices outside of $B_{r+1}(v)$ are not in $B_r(s)$ for any $s\in S=\{v,u_1,u_2,\ldots, u_{2n}\}$ and all vertices inside of $B_{r-1}(v)$ are in $B_r(s)$ for all $s\in S$.

  Next, let $C$ be an $r$-identifying code for $L_n$.  For $s,s'\in S$ with $s\neq s'$, we must have $I_r(s)\triangle I_r(s')\subset B_{r+1}(v)\setminus B_{r-1}(v)$.  Let $K(s)=I_r(s)\cap (B_{r+1}(v)\setminus B_{r-1}(v))$.  We claim for $K(s)\neq K(s')$.  Suppose not.  Then $I_r(s) = K(s) \cup (C\cap B_{r-1}(v)) = I_r(s')$ and so they are not distinguishable.  Hence, $K(s)$ must be distinct for each $s\in S$.  Since the minimum number of elements of a set to produce $2n+1$ distinct subsets is $\lceil\log_2(2n+1)\rceil$, there must be $\lceil\log_2(2n+1)\rceil$ codewords in $B_{r+1}(v)\setminus B_{r-1}(v)$.  We refer to the methods used by Charon, Honkala, Hudry and Lobstein \cite{Charon2001} to show this gives the lower bound:

  $$\frac{\lceil \log_2(2n+1)\rceil}{b_{r+1}^{(n)} - b_{r-1}^{(n)}}.$$

  It is easy to check that $b_r^{(n)}$ is the number of solutions in integers to \begin{equation}\label{ineq:ballsize}|x_1|+|x_2|+\cdots + |x_n| \le r\end{equation} and so $b_{r+1}^{(n)}-b_{r-1}^{(n)}$ is the number of solutions to $$|x_1|+|x_2|+\cdots + |x_n| = k$$ where $k=r$ or $k=r+1$.  Since the number of solutions to $x_1+x_2+\cdots+x_n=k$ is known to be $\binom{n+k-1}{n-1}$, this gives us an upper bound
  \begin{eqnarray*}
     b_{r+1}^{(n)}-b_{r-1}^{(n)} &\le& 2^n\left(\binom{n+r-1}{n-1}+ \binom{n+r}{n-1}\right) \\
     &\le& 2^n\left(\frac{(r+n-1)^{n-1}}{(n-1)!}+ \frac{(r+n)^{n-1}}{(n-1)!}\right) \\
     &=& \frac{2^{n+1}r^{n-1} + p_{n-2}(r)}{(n-1)!}
  \end{eqnarray*}
  which comes from choosing each term to be either positive or negative and then using a standard binomial inequality.  Plugging this in gives us the result described in the theorem.
\end{proof}



\begin{theorem}\label{thm:upperboundeventheorem}
    If $n$ is odd, $0\le k<n+1$, $r\ge n+2$, and $r\equiv k\pmod{(n+2)}$ then $$\mathcal{D}(L_n,r)\le \frac{(n+2)^{n-1}}{2^{n}(r-k)^{n-1}}. $$
    If $n$ is even, $0\le k<(n+2)/2$, $r\ge (n+2)/2$, and $r\equiv k\pmod{(n+2)/2}$ then $$\mathcal{D}(L_n,r)\le \frac{(n+2)^{n-1}}{2^{n}(r-k)^{n-1}}. $$
\end{theorem}
\begin{proof}
  Let $2r_0$ be divisible by $n+2$ and let $k=2r_0/(n+2)$.  We wish to find an $r$-identifying code for $r\ge r_0$.  We define a code $$C=\left\{\left(kx_1,kx_2,\ldots, kx_{n-1},\ell\right): x_1 + x_2 + \cdots + x_{n-1} \equiv 1 \pmod 2 \right\}.$$  Further, let $$S=\left\{\left(kx_1,kx_2,\ldots, kx_{n-1},\ell\right): x_1 + x_2 + \cdots + x_{n-1} \equiv 0 \pmod 2 \right\}.$$
  $C$ will be our code and $S$ will serve as a set of reference points which we will use later.

  We first wish to calculate the density $C\cup S$.  This is simply a tiling of $\Z^{n}$ by the region $[0,k-1]^{n-1}\times\{0\}$ which has only a single codeword in it.  Hence, the density of $C\cup S$ is $1/k^{n-1}=(n+2)^{n-1}/(2^{n-1}r_0^{n-1})$.  Then $C$ is half this density, which is the density stated in the theorem.

  Next, we wish to show that $C$ is an $r$-identifying code for $r\ge r_0$.  Let $e^{(i)}$ represent the vector with a 1 in the $i$th coordinate and a 0 in all other coordinates.  For any vertex $u$, let $u_j$ denote the value of the $j$th coordinate of $u$.

  For $s\in S$, we define the \emph{corners} of $s$ to be the codewords $c$ of the form $c=s \pm ke^{(i)}$ for some $1\le i\le n-1$.

  The remainder of the proof consists of 3 steps:

  \begin{enumerate}
    \item Each vertex $v\in V(G)$ has distance at most $nk/2$ from some $s\in S$ and $v$ has distance at most $r$ from each of the corners of $s$ (in addition, this shows that $I_r(v)$ is nonempty).
    \item If $v=(\mathbf{v},\ell)$, we can uniquely determine $\ell$ from $I_{r}(v)$.  Furthermore, if $c=(\mathbf{c},\ell)\in I_{r}(v)$, we can determine $d(v,c)$.
    \item If $v=(v_1,\ldots, v_{n-1},\ell)$, we can uniquely determine $v_i$ from $I_{r}(v)$ for each $i$.  Thus, $v$ is distinguishable from all other vertices in the graph.
  \end{enumerate}

  \textbf{Step 1}:  Let $v=(v_1,v_2,\ldots, v_{n-1}, \ell)$.  Without loss of generality, we may assume that $(v_1,v_2,\ldots, v_{n-1})\in [0,k]^{n-1}$.    For $i=1,2,\ldots,n-2$ define $$a_i=\left\{\begin{array}{cc}
                                           0 & \text{if $v_i\le k/2$} \\
                                           k & \text{if $v_i> k/2$}
                                         \end{array}
  \right..$$  We then see that $|v_i-a_i|\le k/2$ in either case.  Now consider the vertices $(a_1,a_2,\ldots, a_{n-2},0,\ell)$ and $(a_1,a_2,\ldots, a_{n-2},k,\ell)$.  One of these is in $S$.  Let $a_{n-1}=0$ if the former is in $S$ and $a_{n-1}=k$ if the latter is in $S$.  Then $|v_{n-1}-a_{n-1}|\le k$.  Hence we have
  \begin{eqnarray*}
    d(v,(a_1,a_2,\ldots,a_{n-2},a_{n-1},\ell)) &=& |v_{n-1}-a_{n-1}| + \sum_{i=1}^{n-2} |v_i-a_i| \\
    &\le& k + (n-2)k/2 =nk/2.
  \end{eqnarray*}

  Let $c$ be a corner of $s=(a_1,a_2,\ldots, a_{n-2},a_{n-1},\ell)$.  Then $$d(v,c) \le d(v,s) + d(s,c) \le nk/2 + k = (n+2)k/2 = r_0\le r.$$

  \textbf{Step 2}: Next, we need to determine the last coordinate of $v$.  Write $v=(\mathbf{v},\ell)$.  Suppose that $c=(\mathbf{c},k)\in I_r(v)$.  We then see that $(\mathbf{c},\ell)\in I(v)$ since $d(v,(\mathbf{c},\ell))\le d(v,c)$.  Writing $d(v,(\mathbf{c},\ell))=d_1\le r$, then we see that $(\mathbf{c},\ell\pm j)\in I(v)$ for $j=0,1,\ldots, r-d_1$.  Hence, these codewords form a path of length $2(r-d_1)+1$. Thus, if $\ell_1=\min\{j:(\mathbf{c},j)\in I(v)\}$ and $\ell_2=\max\{j:(\mathbf{c},j)\in I(v)\}$, it follows that $$\ell = \frac{\ell_1+\ell_2}{2}.$$  Furthermore, this tells us once we know $\ell$, we can determine the distance between $v$ and $c$ to be $r - (\ell_2-\ell)$.

  \textbf{Step 3}:  Finally, from Step 1 we know that there is some vertex $s\in S$ such that the codewords $s\pm ke^{(i)}\in I_r(v)$ for each $i$, $1\le i \le n-1$.  Thus, for each $i$ we are guaranteed that there are $m\ge 2$ codewords $c^{(0)},\ldots, c^{(m-1)}$ such that $c^{(j)}=c^{(0)} + 2kje^{(i)}$ and $c^{(j)}\in I_r(v)$ for each $j$.

  Now let $$D=\sum_{{p=1}\atop {p\neq i}}^{n-1} |v_p - c_{p}^{(0)}|.$$  We then see that $d(v,c^{(j)}) = |v_i-c_i^{(j)}| + D$ which is minimized by minimizing $|v_i-c_i^{(j)}|$.  Furthermore, the expression $|v_i-x|$ is unimodal and so the two smallest values of $|v_i-c_i^{(j)}|$ must happen for consecutive integers and they must be amongst our aforementioned $m$ codewords.  Let $a=c^{(\ell)}$ and $b=c^{(\ell+1)}$ be these codewords.  It is easy to check that $a_i\le v_i \le b_i$ by considering evenly spaced point plotted along  the graph of $f(x)=|v_i-x|$.

    This gives
  \begin{eqnarray*}
    d(v,a) &=& v_i-a_i + D \\
    d(v,b) &=& b_i-v_i + D
  \end{eqnarray*}
  Since $a$ and $b$ are codewords, the distances listed above are all known quantities from Step 2.  Subtracting the second line from the first and solving for $v_i$ gives: $$v_i=\frac{d(v,a)-d(v,b)+a_i+b_i}{2}.$$

  Since these are all known quantities, we can compute $v_i$, completing step 3.  Finally, we get the values described in the theorem by taking $r_0$ to be the largest integer smaller than $r$ satisfying the condition that $2r_0/(n+2)$ is an integer, completing the proof.
\end{proof}

\section{Acknowledgements}
I would like to acknowledge the support of a National Science Foundation grant.  Part of this research was done as a research assistant under an NSF grant with co-PIs Maria Axenovich and Ryan Martin.  I would also like to thank  Jonathan D.H. Smith and Cliff Bergman for their help with this paper.  In addition, I would like to thank the referees from a previous submission whose help made this paper much better than it would have been otherwise.

\section{Conclusions}
It is worth noting that the lower bound given in Theorem~\ref{thm:lowerboundtheorem} can only be evaluated as $r\rightarrow\infty$ and not as $n\rightarrow\infty$ since the polynomial in the denominator is a polynomial in $r$, but the coefficients depend on $n$.  However, for fixed $n$ we can make a comparison of the bounds by taking the ratio of the upper bound to the lower bound.  This gives:
\begin{eqnarray*}
 &&\frac{(n+1)^{n-1}}{2^{n}r^{n-1}}\left/ \frac{(n-1)!\lceil \log_2(2n+1)\rceil}{2^{n+1}r^{n-1} + o(r^{n-1})}\right. \\
 &=& \frac{2^{n+1}r^{n-1} + o(r^{n-1})}{2^nr^{n-1}}\cdot\frac{(n+1)^{n-1}}{(n-1)!\lceil\log_2(2n+1)\rceil}\\
 &\approx& (2+o(1))\cdot\frac{(n+1)^{n-1}}{(n-1)^{n-1}}\cdot\frac{e^{n-1}}{\sqrt{2\pi n}\lceil\log_2(2n+1)\rceil} \\
  &\approx& \frac{2e^{n+1}}{\sqrt{2\pi n}\lceil\log_2(2n+1)\rceil}\\
\end{eqnarray*}
and so our lower bound differs from our upper bound by slightly less than a multiplicative factor of $e^n$ when $r\gg n\gg0$.

\bibliographystyle{plain}

\begin{thebibliography}{1}

\bibitem{Ben-Haim2005}
Yael Ben-Haim and Simon Litsyn.
\newblock Exact minimum density of codes identifying vertices in the square
  grid.
\newblock {\em SIAM J. Discrete Math.}, 19(1):69--82 (electronic), 2005.

\bibitem{Charon2001}
Ir{\`e}ne Charon, Iiro Honkala, Olivier Hudry, and Antoine Lobstein.
\newblock General bounds for identifying codes in some infinite regular graphs.
\newblock {\em Electron. J. Combin.}, 8(1):Research Paper 39, 21 pp.
  (electronic), 2001.

\bibitem{Charon2002}
Ir{\`e}ne Charon, Olivier Hudry, and Antoine Lobstein.
\newblock Identifying codes with small radius in some infinite regular graphs.
\newblock {\em Electron. J. Combin.}, 9(1):Research Paper 11, 25 pp.
  (electronic), 2002.

\bibitem{Cohen1997}
G{\'e}rard Cohen, Iiro Honkala, Simon Litsyn, and Antoine Lobstein.
\newblock {\em Covering codes}, volume~54 of {\em North-Holland Mathematical
  Library}.
\newblock North-Holland Publishing Co., Amsterdam, 1997.

\bibitem{Cohen2001}
G{\'e}rard~D. Cohen, Iiro Honkala, Antoine Lobstein, and Gilles Z{\'e}mor.
\newblock On codes identifying vertices in the two-dimensional square lattice
  with diagonals.
\newblock {\em IEEE Trans. Comput.}, 50(2):174--176, 2001.

\bibitem{Kabatyanskiui1988}
G.~A. Kabatyanski{\u\i} and V.~I. Panchenko.
\newblock Packings and coverings of the {H}amming space by unit balls.
\newblock {\em Dokl. Akad. Nauk SSSR}, 303(3):550--552, 1988.

\bibitem{Karpovsky1998}
Mark~G. Karpovsky, Krishnendu Chakrabarty, and Lev~B. Levitin.
\newblock On a new class of codes for identifying vertices in graphs.
\newblock {\em IEEE Trans. Inform. Theory}, 44(2):599--611, 1998.

\bibitem{Stanton2011a}
Brendon Stanton.
\newblock On Vertex Identifying Codes for Infinite Lattices
\newblock {\em PhD thesis}, Iowa State University, 2011.

\end{thebibliography}

\end{document}